\documentclass[a4paper,12pt, reqno]{amsart}
\usepackage{amssymb,amsthm,amsmath}
\usepackage{cite}
\usepackage{mathabx}
\usepackage{mathtools}
\usepackage{amsmath,amsthm,amscd,amssymb}
\usepackage{latexsym}
\usepackage[colorlinks,citecolor=blue,pagebackref,hypertexnames=false]{hyperref}
\numberwithin{equation}{section}
\theoremstyle{plain}
\newtheorem{rem}{Remark}
\newtheorem{theorem}{Theorem}[section]

\newtheorem{corollary}[theorem]{Corollary}
\newtheorem{proposition}[theorem]{Proposition}

\theoremstyle{definition}
\newtheorem{definition}[theorem]{Definition}

\theoremstyle{remark}

\newtheorem{case[theorem]}{Case}

\def\P{{\mathcal{P_{\textrm{aff}}}}}
\def\v{{\vartheta}}
\def \R{{\mathbb R}}

\def \F{{\mathcal F}}

\def\U{{\mathcal U}}

\def\norm#1.#2.{\lVert#1\rVert_{#2}}

\def\R{\mathbb R}

\title[Weyl Transform on Nonunimodular Groups]
{ Weyl Transform on Nonunimodular Groups
}

\author{Santosh Kumar Nayak}\address{Santosh Kumar Nayak  \endgraf School of Mathematical Sciences	\endgraf National Institute of Science Education and Research\endgraf Bhubaneswar - 752 050, India.} \email{nayaksantosh212@gmail.com, mathnayak@gmail.com}

\keywords{Representation theory, Nonunimodular group, Fourier Transform, Affine group, Similitude group, Affine Poincar\'e group, Weyl transform} \subjclass{Primary 47G10, 47G30,43A15, 43A30, Secondary 42B35}
\date{\today}
\begin{document}
	\thanks{S. K. Nayak is supported by the Institute Fellowship.}
	
	\maketitle

	\allowdisplaybreaks
	
	\begin{abstract}
        For $p>2$, B. Simon\cite{Simon} studied the unboundedness of the Weyl transform  for symbol belonging to $L^p({\R^n\times \R^n})$.
		In this article, we study the analog of   unboundedness of the Weyl transform on some nonunimodular groups, namely, the affine group, similitude group, and affine Poincar\'e group.
	\end{abstract}

	
	\section{Introduction}

Weyl transform is an operator introduced by H. Weyl in 1950, \cite{Weyl}. It is a type of self-adjoint pseudo-differential operator on $\R^n$, \cite{stien}. H. Weyl studied this operator while solving quantization problems in quantum mechanics, \cite{Kohn}. Physicists always love to work with self-adjoint operators. The theory of the Weyl transforms covers a broad area of great interest in both mathematical analysis and physics. In a number of problems like elliptic theory, regularity problems, spectral asymptotic, etc., Weyl transforms have proved to be a useful tool, \cite{wong1}. The Weyl transform(operator) has been deeply investigated mainly in the case where the symbol is a smooth function on $\R^n\times\R^n$ belonging to some symbol classes. For more details, we refer to \cite{Boggiatto1,Boggiatto2,Toft1, Toft2}. The Weyl transform is an operator acting on $L^2(\R^n)$ that belongs to Hilbert-Schmidt class when the symbol lies in $L^2(\R^n\times\R^n)$, \cite{Weyl}. And when the symbol $\sigma $ lies in $L^p(\R^n\times\R^n), 1\leq p\leq 2,$ the Weyl transform is bounded on $L^2(\R^n)$. For $p>2$, B. Simon\cite{Simon} studied the unboundedness of the Weyl transform  for symbol belonging to $L^p({\R^n\times \R^n})$. The proof of boundedness and unboundedness of the Weyl transform can be found in \cite{wong1}.

 In this paper, we assume group G to be any of the three groups, namely, affine group, similitude group, and affine Poincar\'e group, and their dual $\widehat{G}$ are discrete sets.

 Let $L^p(G\times\widehat{G},S_p)$ be the set of all $S_p$ valued functions such that 
$$\|f\|_{p,\mu}^p=\sum_{\rho\in\widehat{G}}\int_{G}\|f(x,\rho)K_{\rho}^{\frac{1}{p}}\|_{S_p}^pd\mu(x)<\infty, \quad 1\leq p<\infty,$$
$$\|f\|_{\infty,\mu}=\text{ess~~sup}_{(x,\rho)\in G\times\widehat{G}}\|f(x,\rho)\|_{S_{\infty}}<\infty.$$

\begin{definition}
Let $f,g\in C_c(G)$. Then the Wigner transform associated to $f$ and $g$ is given by 
$$W(f,g)(x,\rho)=\int_G f(x^{\prime})\tau_{x^{\prime}}g(x)\rho(x^{\prime})d\mu(x^{\prime}),$$ where $\tau_{x}f(y)=f(x^{-1}y)$ is the left translation operator on $L^p(G)$.
\end{definition}
 Proof of the  following results is followed step by step from \cite{Ghosh}. So we omit those proofs from the article but to make it self-contained, we just have stated.
\begin{theorem}
Let $f,g\in C_c(G)$. Then for $2\leq p\leq\infty$, we have $W(f,g)\in L^p(G\times\widehat{G},S_p)$ and 
$$\|W(f,g)\|_{p,\mu}\leq \|f\|_{L^2(G)}\|g\|_{L^2(G)}.$$
\end{theorem}
The following proposition represents the Fourier transform in terms of the Wigner transform.
\begin{proposition}\label{wigner}
Let $f\in L^1(G)\cap L^2(G)$ and $C=\int\limits_{G}g(y)d\mu(y)\neq 0$. Then $$\widehat{f}(\rho)=C^{-1}\int_{G}W(f,g)(y,\rho)d\mu(y),$$ for all $\rho\in\widehat{G}.$
\end{proposition}
\begin{corollary}
Let $f\in L^1(G)\cap A(G)$ and $g\in L^1(G)\cap L^2(G)$ with $C=\int\limits_{G}g(x)d\mu(x)\neq 0$. Then 
$$f(x)=C^{-1}\sum_{\rho\in\widehat{G}}\operatorname{Tr}\left(\rho^{\ast}(x)\left(\int_G W(f,g)(x^{\prime},\rho)d\mu(x^{\prime})\right)K_{\rho}\right).$$
\end{corollary}
\begin{definition}\label{Weyl}
Let $\sigma\in L^p(G\times \widehat{G}, S_p), 1\leq p\leq 2.$ Then the Weyl transform $W_{\sigma}:L^2(G)\to L^2(G)$ is defined by 
\begin{equation}\label{Weyl1}
\langle W_{\sigma}f,\overline{g}\rangle=\langle W(f,g),\sigma\rangle_{\mu}=\sum_{\rho\in\widehat{G}}\int_{G}\operatorname{Tr}\left(\sigma(x,\rho)^{\ast}W(f,g)(x,\rho)K_{\rho}\right)d\mu(x),
\end{equation}where $f,g$ are in $L^2(G)$ and $\langle,\rangle_{\mu}$ denotes innerproduct in $L^2(G\times\widehat{G},S_2)$.
\end{definition}
\begin{theorem}\label{bounded}
Let $\sigma\in L^p(G\times\widehat{G}, S_p), 1\leq p\leq 2$. Then $$\|W_{\sigma}\|\leq \|\sigma\|_{p,\mu}.$$
\end{theorem}
For $G=\R^n$, B. Simon \cite{Simon} gave the example of symbols in $L^p(\R^n\times\R^n),p>2$, such that the corresponding Weyl transform becomes unbounded. This idea has been motivating many authors to study the unboundedness of Weyl transform in other settings like the Heisenberg group \cite{Peng}, quaternion Heisenberg group \cite{Chen}, upper half plane \cite{Peng1}, motion groups \cite{Ghosh}, reduced Heisenberg group with multidimensional center \cite{San4}, and other Weyl transforms can be found in \cite{Zhao,Rachdi}. In this article, we study the unboundedness of Weyl transform on the affine group $U$, similitude group $\mathbb{SIM}(2)$, and affine Poincar\'e group $\P$.

The organization of the paper is as follows: we investigate the unboundedness of Weyl transform on the affine group, similitude group, and affine Poicar\'e group, respectively in Section \ref{s2}, Section \ref{s3}, and Section \ref{s4}. 
\section{Affine Group}\label{s2}
In this section, we consider the group $G$ to be an affine group $U$. Here we recall the Fourier analysis of the affine group, $U$, in detail. In this setup, we will define the Weyl transform formula and show that it cannot be extended as a bounded operator for the symbol in the corresponding $L^p$ spaces with $2<p<\infty$.

Let $U=\{(b,a):b\in\mathbb{R}, a>0 \}$ be the upper half plane. Then  with respect to the   binary operation `$\cdot$' on $U$,  defined as
\begin{equation}\label{Binop} 
	(b,a)\cdot (c,d)=(b+ac,ad),\quad (b,a), (c,d)\in U,
\end{equation} 
 the upper half plane $U$  forms  a non-abelian group. It can be shown  $(\frac{-b}{a},\frac{1}{a})$ is the inverse element of $(b,a)$ and $(0,1)$ is the identity element in $U$. The left and right Haar measures on $U$ are  $d\mu=\frac{dbda}{a^2}$ and $d\gamma = \frac{dbda}{a}$ respectively.
With respect to the above multiplication `$\cdot$' defined by \eqref{Binop}, U is also a locally compact and Hausdroff group on which the left Haar measure is different from the right Haar measure, and hence $U$ is a non-unimodular group.

Let $\mathcal{U}(L^2(\mathbb{R}_{\pm}))$  be the set of all unitary operators on $L^{2}(\mathbb{R}_{\pm})$. Then the unitary irreducible representations of $U$ on $L^2(\mathbb{R}_{\pm})$ are given by the mapping $\rho_{\pm}: U\rightarrow \mathcal{U}(L^{2}(\mathbb{R}_{\pm}))$ defined as

$$(\rho_{+}(b,a)u)(s)=a^{1/2}e^{-ibs}u(as), \quad s\in \mathbb{R}_{+}=[0,\infty),$$
for every $u\in L^{2}(\mathbb{R}_{+})$, and  
$$(\rho_{-}(b,a)u)(s)=a^{1/2}e^{-ibs}v(as),  \quad s\in \mathbb{R}_{-}=(-\infty,0],$$ 
for every $v\in L^{2}(\mathbb{R}_{-}).$  These two $\{\rho_+,\rho_-\}$ are only the irreducible inequvalent unitary representations, i.e., $\widehat{U}=\{\rho_+,\rho_-\}$.
 The unbounded operators, $K_{\pm}$, on $L^{2}(\mathbb{R}_{\pm}), $ are given by
\begin{equation}
	(K_{\pm}\phi)(s)=|s|\phi(s), \quad s\in \mathbb{R}_{\pm},
\end{equation} are known as Duflo-Moore operators, \cite{Fuhr}.

Let $f\in L^1(U)\cap L^2(U)$, the Fourier transform
$$\widehat{f}(\rho_{\pm})=\int_{U}f(b,a)\rho_{\pm}(b,a)\dfrac{dbda}{a^2}.$$ Let $\phi\in L^2(\R_{\pm})$, then $\widehat{f}(\rho_{\pm})\phi(s)=\int\limits_0^{\infty}\int\limits_{-\infty}^{\infty}f(b,a)\left( \rho_{\pm}(b,a)\phi\right)(s)\dfrac{dbda}{a^2},\phi\in L^2(\R_{\pm})$. Define $L^p$-Fourier transform $\F_p(f)\rho_{\pm}=\widehat{f}(\rho_{\pm})K_{\pm}^{\frac{1}{p^{\prime}}}, 1\leq p\leq \infty,\frac{1}{p}+\frac{1}{p^{\prime}}=1.$
\begin{theorem}[Fourier inversion formula]\cite{San1}
Let $f\in L^1(U)$. Then 
$$f(x)=\sum_{j\in\{\pm\}} \operatorname{Tr}\left(\rho_{j}(b,a)^{\ast}\widehat{f}(\rho_j)K_{j}\right).$$
\end{theorem}
\begin{theorem}[Plancherel Formula]\cite{San1}
Let $f\in L^1(U)\cap L^2(U)$. Then 
$$\int_U |f(b,a)|^2\dfrac{dbda}{a^2}=\sum_{j\in\{\pm\}}\|\widehat{f}(\rho_{j})K_j^{\frac{1}{2}}\|_{S_2}^2.$$
\end{theorem}
For $G=U$, recall the definition of Weyl transform defined in \eqref{Weyl1}, and the Theorem \ref{bounded} says that it is bounded when $1\leq p \leq 2$. Now we prove the unboundedness of Weyl transform when $p>2$.
\begin{theorem}\label{unbounded1}
For $2<p<\infty$, there exists an operator valued function $\sigma$ in $L^p(U\times \widehat{U}, S_p)$ such that $W_{\sigma}$ is not a bounded linear operator on $L^2(U)$. 
\end{theorem}
We prove the Theorem \ref{unbounded1} using the method of contrapositive in three propositions.
\begin{proposition}\label{prop1}
Let $2<p<\infty$ and $\frac{1}{p}+\frac{1}{p^{\prime}}=1$. Then for all $\sigma\in L^p(U\times\widehat{U}, S_p)$, the Weyl transform $W_{\sigma}$ is bounded linear operator on $L^2(U)$ iff there exists a constant $C$ such that $$\|W(f,g)\|_{p^{\prime},\mu}\leq C\|g\|_{L^2(U)}\|f\|_{L^2(U)},$$ for all $f,g$ in $L^2(U)$.
\end{proposition}
\begin{proof}
Assume that $$\|W(f,g)\|_{p^{\prime},\mu}\leq C\|g\|_{L^2(U)}\|f\|_{L^2(U)},\quad f,g\in L^2(U),$$ for some $C>0$. By the definition \ref{Weyl}, we obtain the necessary part of the theorem. Conversely, suppose that for each $\sigma\in L^p(U\times\widehat{U}, S_p)$, $W_{\sigma}$ is bounded, i.e.,
$$\|W_{\sigma}f\|_{L^2(U)}\leq C({\sigma})\|f\|_{L^2(U)},$$ for all $f\in L^2(U)$. Define a family of linear operators $M_{f,g}:L^p(U\times\widehat{U}, S_p)\to \mathbb{C}$ by $$M_{f,g}(\sigma)=\langle W_{\sigma}f,g\rangle,\quad f,g\in C_c(U).$$ For each $f,g\in C_c(U)$, it is easy to check that this family is pointwise bounded linear functional. By uniform bounded principle, there exists $C$ such that $\|M_{f,g}\|\leq C$, for all $ f,g\in C_c(U), \|f\|_2=\|g\|_2=1$. Hence $|\langle W_{\sigma}f,g\rangle|\leq C\|\sigma\|$. Thus for $f,g\in C_c(U)$,
\begin{align*}
    \|W(f,g)\|_{p^{\prime},\mu}&=\sup_{\|\sigma\|_{p,\mu}=1}|\langle W(f,g),\sigma\rangle_{\mu}|\\
    &=\sup_{\|\sigma\|_{p,\mu}=1}|\langle W_{\sigma}f,\overline{g}\rangle|\leq C\|f\|_2\|g\|_2.
\end{align*}
By density argument, we conclude the proposition.

\end{proof}
\begin{proposition}\label{prop2}
Let $2<p<\infty$ and $f$ be a square integrable, compactly supported function on $U$ such that $\int\limits_{U}f(b,a)d\mu(b,a)\neq 0$. If $W_{\sigma}$ is a bounded operator on $L^2(U)$ for all $\sigma$ in $L^p(U\times\widehat{U},S_p)$, then 
$$\|\F_p(f)\rho_{+}\|_{S_{p^{\prime}}}^{p^{\prime}}+\|\F_p(f)\rho_{-}\|_{S_{p^{\prime}}}^{p^{\prime}}<\infty.$$
\end{proposition}
\begin{proof}
From Proposition \ref{wigner}, it is enough to show
$$\sum_{j\in\{\pm\}}\|\int_{U}W(f,f)(b,a,\rho_j)K_{j}^{\frac{1}{p^{\prime}}}d\mu(b,a)\|_{S_{p^{\prime}}}^{p^{\prime}}<\infty,$$ for all square integrable, compactly supported function $f$ on $U$ such that $\int\limits_{U}f(b,a)d\mu(b,a)\neq 0$. Suppose $f$ is supported in $A$ of $U$, then $W(f,f)$ is supported in $AA\times \widehat{U}$.
Now by H\"older's inequality, Minkowski's integral inequality, and Proposition \ref{prop1}, we get
\begin{align*}
  &\sum_{j\in\{\pm\}}\|\int_{U}W(f,f)(b,a,\rho_j)K_{j}^{\frac{1}{p^{\prime}}}d\mu(b,a)\|_{S_{p^{\prime}}}^{p^{\prime}}\\
  &\leq \left(\int_{AA}\left(\sum_{j\in\{\pm\}}\|W(f,f)(b,a,\rho_j)K_{j}^{\frac{1}{p^{\prime}}}\|_{S_{p^{\prime}}}^{p^{\prime}}\right)^{\frac{1}{p^{\prime}}} d\mu(b,a)\right)^{p^{\prime}}\\
  & \leq \left(\int_{AA}d\mu(b,a)\right)^{\frac{p^{\prime}}{p}}\int_{AA}\sum_{j\in\{\pm\}}\|W(f,f)(b,a,\rho_j)K_{j}^{\frac{1}{p^{\prime}}}\|_{S_{p^{\prime}}}^{p^{\prime}}d\mu(b,a)<\infty.
\end{align*}
This completes the proof.
\end{proof}
\begin{proposition}\label{example1}
For $p\in (2,\infty)$, does there exists a square-integrable, compactly supported function $f$ on $U$ with $\int\limits_{U} f(b,a)\dfrac{dbda}{a^2}\neq 0$ such that $\|\F_p(f)\rho_{+}\|_{S_{p^{\prime}}}^{p^{\prime}}+\|\F_p(f)\rho_{-}\|_{S_{p^{\prime}}}^{p^{\prime}}=\infty?$
\end{proposition}
 The answer to this question is yes. We will certainly find a function $f_{\alpha}, 0<\alpha<\frac{1}{2}$, at the end of the section, which is needed for Proposition \ref{example1}.
Let $f\in L^2(U)$. Now for all $\phi\in L^2(\R_{+})$,
\begin{align*}
    \left(\widehat{f}(\rho_{+})K_+^{\frac{1}{p^{\prime}}}\phi\right)(s)&=\int_0^{\infty}\int_{-\infty}^{\infty}f(b,a)\left(\rho_{+}(b,a)K_{+}^{\frac{1}{p^{\prime}}}\phi\right)(s)\frac{dbda}{a^2}\\
    &= \int_0^{\infty}\int_{-\infty}^{\infty}f(b,a)a^{\frac{1}{2}}e^{-ibs}\left(K_{+}^{\frac{1}{p^{\prime}}}\phi\right)(as)\frac{dbda}{a^2}\\
    &= \int_0^{\infty}(\F_1f)(s,a)a^{\frac{1}{2}}|as|^{\frac{1}{p^{\prime}}}\phi(as)\frac{da}{a^2}\\
    &= \int_0^{\infty}(\F_1f)(s,a)a^{\frac{1}{2}+{\frac{1}{p^{\prime}}}}|s|^{\frac{1}{p^{\prime}}}\phi(as)\frac{da}{a^2}.
\end{align*}
Now substituting $as=t$, we get $da=\frac{dt}{s}$. 
\begin{align*}
 \left(\widehat{f}(\rho_{+})K_+^{\frac{1}{p^{\prime}}}\phi\right)(s) &= \int_0^{\infty}\left(\F_1f\right)\left(s,\frac{t}{s}\right)\frac{t}{s}^{\frac{1}{2}+{\frac{1}{p^{\prime}}}}|s|^{\frac{1}{p^{\prime}}}\frac{s^2}{t^2}\phi(t)\frac{dt}{s}.
\end{align*}
Again for all $\phi\in L^2(\R_{-})$, we have
$$\left(\widehat{f}(\rho_{-})K_{-}^{\frac{1}{p^{\prime}}}\phi\right)(s) = \int_{-\infty}^0\left(\F_1f\right)\left(s,\frac{t}{s}\right)\frac{t}{s}^{\frac{1}{2}+{\frac{1}{p^{\prime}}}}|s|^{\frac{1}{p^{\prime}}}\frac{s^2}{t^2}\phi(t)\frac{dt}{s}.$$
Now \begin{align*}
    \|\widehat{f}(\rho_{+})K_{+}^{\frac{1}{p^{\prime}}}\|_{S_2}^2+\|\widehat{f}(\rho_{-})K_{-}^{\frac{1}{p^{\prime}}}\|_{S_2}^2&=\int_0^{\infty}\int_0^{\infty}\left|\left(\F_1f\right)\left(s,\frac{t}{s}\right)\right|^2\frac{t}{s}^{1+{\frac{2}{p^{\prime}}}}|s|^{\frac{2}{p^{\prime}}-2}\frac{s^4}{t^4}dsdt\\
    &+\int_{-\infty}^0\int_{-\infty}^0 \left|\left(\F_1f\right)\left(s,\frac{t}{s}\right)\right|^2\frac{t}{s}^{1+{\frac{2}{p^{\prime}}}}|s|^{\frac{2}{p^{\prime}}-2}\frac{s^4}{t^4}dsdt.
\end{align*}
 Substituting $\frac{t}{s}=a$, then we get $\frac{dt}{s}=da$. 
 \begin{align}\label{S2 Norm}
   \|\widehat{f}(\rho_{+})K_{+}^{\frac{1}{p^{\prime}}}\|_{S_2}^2+\|\widehat{f}(\rho_{-})K_{-}^{\frac{1}{p^{\prime}}}\|_{S_2}^2&=\int_0^{\infty}\int_0^{\infty}\left|\left(\F_1f\right)\left(s,a\right)\right|^2a^{1+{\frac{2}{p^{\prime}}}}|s|^{\frac{2}{p^{\prime}}-1}\frac{1}{a^4}dsda\nonumber\\
    &+\int_{-\infty}^0\int_0^{\infty} \left|\left(\F_1f\right)\left(s,a\right)\right|^2a^{1+{\frac{2}{p^{\prime}}}}|s|^{\frac{2}{p^{\prime}}-1}\frac{1}{a^4}dsda \nonumber \\
    &= \int_{-\infty}^{\infty}\int_0^{\infty} \left|\left(\F_1f\right)\left(s,a\right)\right|^2a^{1+{\frac{2}{p^{\prime}}}-4}|s|^{\frac{2}{p^{\prime}}-1}dsda.
 \end{align}
For $0<\alpha<\frac{1}{2}$, let us choose  $$f_{\alpha}(b,a)= 
\begin{cases}
|b|^{-\alpha}a^2,\quad (b,a)\in Q\\
0, \quad\quad \text{otherwise},
\end{cases}
$$
where $Q=\{(b,a):-p\leq b,a \leq p\},p>0$. Hence $$(\F_1f_{\alpha})(s,a)=\left[ (2\pi)^{-1/2}\int_{-p}^p e^{-ibs}|b|^{-\alpha}db\right]a^2.$$
Now for $s>0$,
$$\int_{-p}^p e^{-ibs}|b|^{-\alpha}db=2\left(\int_0^{s p}t^{-\alpha}\cos tdt\right)s^{-1+\alpha},\quad j=1,2.$$ When $\alpha\in (0,\frac{1}{2})$, then $|\int\limits_0^{s p}t^{-\alpha}\cos tdt|\geq B,$ for some $B>0$. Now 
\begin{align*}
 \int_{-\infty}^{\infty}\int_0^{\infty} \left|\left(\F_1f\right)\left(s,a\right)\right|^2a^{1+{\frac{2}{p^{\prime}}}-4}|s|^{\frac{2}{p^{\prime}}-1}dsda&\geq\dfrac{2B^2}{\pi}\int_0^p\int_R^{\infty}  s^{2(-1+\alpha)+\frac{2}{p^{\prime}}-1}a^{1+{\frac{2}{p^{\prime}}}}dsda.
\end{align*}
But $$\int_R^{\infty}s^{2(-1+\alpha)+{\frac{2}{p^{\prime}}-1}}ds=\infty,$$ when $\alpha>1-\frac{1}{p^{\prime}}$. Also $f_{\alpha}$ is square integrable if $\alpha<\frac{1}{2}$. Let us choose $\frac{1}{2}>\alpha>1-\frac{1}{p^{\prime}}$, then the Equation \eqref{S2 Norm} becomes
$$\|\widehat{f_{\alpha}}(\rho_{+})K_{+}^{\frac{1}{p^{\prime}}}\|_{S_2}^2+\|\widehat{f_{\alpha}}(\rho_{-})K_{-}^{\frac{1}{p^{\prime}}}\|_{S_2}^2=\infty.$$ Hence for given $f_{\alpha},\frac{1}{2}>\alpha>1-\frac{1}{p^{\prime}}$, we get $$\|\F_p(f_{\alpha})\rho_{+}\|_{S_{p^{\prime}}}^{p^{\prime}}+\|\F_p(f_{\alpha})\rho_{-}\|_{S_{p^{\prime}}}^{p^{\prime}}=\infty, 1<p^{\prime}<2.$$
\section{Similitude Group}\label{s3}
This section details the Fourier analysis of the similitude group, $\mathbb{SIM}(2)$. In this setup, we will define the Weyl transform formula and show that it cannot be extended as a bounded operator for the symbol in the corresponding $L^p$ spaces with $2<p<\infty$.

The similitude group, $\mathbb{SIM}(2)$, can be considered as the generalization of the affine group. This group contains affine group as a subgroup.  Also one can consider $\mathbb{SIM}(2)$ as a complexification of the affine group, $U$. This group arises in the study of $2$-dimensional wavelet transforms. It is a four-parameter group that contains the translations in the image plane, $\mathbb{R}^2$, global dilations (zooming in and out by $a>0$), and rotations around the origin ($\theta\in [0,2\pi)$). The action on the plane is given by,
$$x=(b,a,\theta)y= aR_{\theta}y+b,\quad x,y\in\R^2,$$ where
$b\in \mathbb{R}^2$, $a\in \mathbb{R}$ and $R_{\theta}$ is the $2\times 2$ rotation matrix,
\begin{equation}
	R_{\theta}=\left(\begin{matrix}
		\cos\theta &  -\sin\theta\\
		\sin\theta   &  \cos\theta
	\end{matrix}\right).
\end{equation}
An useful representation of  $(b,a,\theta)$ is a $3\times 3$ matrix given by
\begin{equation}
	(b,a,\theta)=\left(\begin{matrix}
		aR_\theta &  b\\
		0^{\intercal}   &  1
	\end{matrix}\right),~~~ 0^{\intercal}=(0,0).
\end{equation}
Matrix multiplication then gives the composition of successive transformations and thus the group law is derived from that,
\begin{eqnarray}
	(b,a,\theta)\ast (b^{\prime},a^{\prime},\theta^{\prime})&=&(b+aR_\theta b^{\prime},aa^{\prime},\theta+\theta^{\prime})\nonumber\\
	e&=&(0,1,0)\nonumber\\
	(b,a,\theta)^{-1}&=&(-a^{-1}R_{-\theta}b,a^{-1},-\theta).\nonumber
\end{eqnarray}
With respect to the operation $\ast$, $\mathbb{SIM}(2)$ is a non-abelian group in which $({0},1,0)$ is the identity element and $(\frac{-1}{a}R_{-\theta}{b},\frac{1}{a},-\theta)$ is the inverse of $({b},a,\theta)$ in $\mathbb{SIM}(2)$. Also, it can be shown that $\mathbb{SIM}(2)$ is a non-unimodular group as its left and right Haar measures  given by $$d\mu_{L}({b},a,\theta)=\dfrac{d{b}dad\theta}{a^3},~~ d\mu_{R}({b},a,\theta)=\dfrac{d{b}dad\theta}{a},$$ respectively, are different.

Moreover, the similitude group $\mathbb{SIM}(2)$, has the structure of a semidirect product:
$$\mathbb{SIM}(2)=\mathbb{R}^{2}\rtimes(\mathbb{R}_{\ast}^{+}\times \mathrm{SO}(2)),$$ where $\mathbb{R}^2$ is the subgroup of the translations, $\mathbb{R}_{\ast}^{+}$ that of dilations and $\mathrm{SO}(2)$ of rotations. Topologically, one can identify $\mathbb{R}^2$ with $\mathbb{C}$, the complex plane and $\mathbb{R}_{\ast}^{+}\times \mathrm{SO}(2)$ with $\mathbb{C}^{\ast},$ the complex plane with the origin is removed. Then $$\mathbb{SIM}(2)=\mathbb{C}\rtimes C^{\ast},$$
consisting of $(z,w),$ where $z\in\mathbb{C}$ and $w\in \mathbb{C}^{\ast},$ the group composition law is,
$$(z_1,w_1)(z_2,w_2)=(z_1+w_1z_2,w_1w_2).$$  In particular, if one considers the elements $(z,w)$, with $z=b+ic,$ $c=0$ and $w=ae^{i\theta},$ $\theta=0$ then these elements clearly constitute a subgroup, which is just the affine group of the line. Thus  the affine group forms a subgroup of $\mathbb{SIM}(2)$.

Let $\U(L^2(\mathbb{R}^2))$ be the set of all unitary operators on $L^2(\R^2)$. 
We denote $\pi :\mathbb{SIM}(2)\rightarrow \U(L^2(\mathbb{R}^2))$ be the mapping of $\mathbb{SIM}(2)$ into the group $\U(L^2(\mathbb{R}^2))$, is given by 
$$ \big(\pi({b},a,\theta)\phi\big)(x)=ae^{-i{b}\cdot x}\phi(aR_{-\theta}x),\quad x\in\mathbb{R}^2,$$
for all $(b,a,\theta)$ in $\mathbb{SIM}(2)$ and all $\phi$ in $L^2(\mathbb{R}^2)$.
Here $\pi$ is the only irreducible unitary representation of $\mathbb{SIM}(2)$ on $L^2(\mathbb{R}^2)$ upto equivalence, \cite{San2}. i.e., dual of $\mathbb{SIM}(2)$ is just singleton.

Let $f\in L^1(\mathbb{SIM}(2))\cap L^2(\mathbb{SIM}(2))$ and define the Fourier transform $\widehat{f}$ of $f$ on $\widehat{\mathbb{SIM}(2)}=\{\pi\}$ by 
$$ (\widehat{f}(\pi)\phi)(x)=\int\limits_{\mathbb{SIM}(2)}f({b},a,\theta)(\pi({b},a,\theta)\phi)(x)\frac{d{b}dad\theta}{a^3},\quad x\in\R^2,$$
for all $\phi\in L^{2}(\mathbb{R}^2).$

 Define the function $K_{\pi}\phi$ on $\mathbb{R}^2$ by, \begin{equation}\label{D-M}
	(K_{\pi}\phi)(x)=\|x\|_2^2\phi(x),\quad x\in \R^2,
\end{equation} where $\|x\|_2=\sqrt{x_1^2+x_2^2}$.
Define the $L^p$-Fourier transform $\F_p(f)\pi=\widehat{f}(\pi)K_{\pi}^{\frac{1}{p^{\prime}}}, 1\leq p\leq \infty,\frac{1}{p}+\frac{1}{p^{\prime}}=1.$ Now recall the Fourier inversion and Plancherel formula for the similitude group from \cite{San2}.
\begin{theorem}{(Inversion theorem)}

Let $f$ be in $L^1(\mathbb{SIM}(2))\cap L^2(\mathbb{SIM}(2))$. Then $$f(b,a,\theta)=\operatorname{Tr}\left(\pi(b,a,\theta)^{\ast}\widehat{f}(\pi)K_{\pi}\right).$$
\end{theorem}
Let $S_r$ be the space of all $r$-Schatten class operators, which is a Banach space with the norm $\|A\|_{S_r}^r=\operatorname{Tr}\left(A^{\ast}A\right)^{r/2}$ for $1\leq p<\infty$, and for $r=\infty$, the norm is the usual operator norm. 
\begin{theorem}{(Plancherel formula)}
Let $f$ be in $L^1(\mathbb{SIM}(2))\cap L^2(\mathbb{SIM}(2))$. Then 
$$\int_{\mathbb{SIM}(2)}|f(b,a,\theta)|^2\dfrac{dbdad\theta}{a^3}=\|\widehat{f}(\pi)K_{\pi}^{1/2}\|_{S_2}^2.$$
\end{theorem}
For $G=\mathbb{SIM}(2)$, recall the definition of Weyl transform defined in \eqref{Weyl1}, and the Theorem \ref{bounded} says that it is bounded when $1\leq p \leq 2$. Now we prove the unboundedness of Weyl transforms when $p>2$ for the similitude group.
\begin{theorem}\label{unbounded2}
For $2<p<\infty$, there exists an operator valued function $\sigma$ in $L^p(\mathbb{SIM}(2)\times \widehat{\mathbb{SIM}(2)}, S_p)$ such that $W_{\sigma}$ is not a bounded linear operator on $L^2(\mathbb{SIM}(2))$. 
\end{theorem}
We prove the Theorem \ref{unbounded2} using the method of contrapositive in three propositions.
\begin{proposition}
Let $2<p<\infty$ and $\frac{1}{p}+\frac{1}{p^{\prime}}=1$. Then for all $\sigma\in L^p(\mathbb{SIM}(2)\times\widehat{\mathbb{SIM}(2)}, S_p)$, the Weyl transform $W_{\sigma}$ is bounded linear operator on $L^2(\mathbb{SIM}(2))$ iff there exists a constant $C$ such that $$\|W(f,g)\|_{p^{\prime},\mu}\leq C\|g\|_{L^2(\mathbb{SIM}(2))}\|f\|_{L^2(\mathbb{SIM}(2))},$$ for all $f,g$ in $L^2(\mathbb{SIM}(2))$.
\end{proposition}
\begin{proof}
The proof follows from Theorem \ref{prop1}.
\end{proof}
\begin{proposition}
Let $2<p<\infty$ and $f$ be a square integrable, compactly supported function on $\mathbb{SIM}(2)$ such that $\int\limits_{\mathbb{SIM}(2)}f(b,a,\theta)d\mu_L(b,a,\theta)\neq 0$. If $W_{\sigma}$ is a bounded operator on $L^2(\mathbb{SIM}(2))$ for all $\sigma$ in $L^p(\mathbb{SIM}(2)\times\widehat{\mathbb{SIM}(2)},S_p)$, then 
$$\|\F_p(f)\pi\|_{S_{p^{\prime}}}^{p^{\prime}}<\infty.$$
\end{proposition}
\begin{proof}
The proof follows from Theorem \ref{prop2}.
\end{proof}
\begin{proposition}\label{example2}
For $p\in (2,\infty)$, does there exists a square-integrable, compactly supported function $f$ on $\mathbb{SIM}(2)$ with $\int\limits_{\mathbb{SIM}(2)} f(b,a,\theta)d\mu(b,a,\theta)\neq 0$ such that $\|\F_p(f)\pi\|_{S_{p^{\prime}}}^{p^{\prime}}=\infty?$
\end{proposition}

The answer to this question is yes. We will certainly find a function $f_{\alpha}, 0<\alpha<\frac{1}{2}$, at the end of the section, which is needed for Proposition \ref{example2}.
The $L^p$-Fourier transform, for $f\in L^1(\mathbb{SIM}(2)\cap L^p(\mathbb{SIM}(2)$, is defined by $$\F_p (f)\pi=\widehat{f}(\pi)K_{\pi}^{\frac{1}{p^{\prime}}},$$ where $(K_{\pi}\phi)(x)=\|x\|_{2}^2\phi(x)$. The integral representation of the operator $\F_p (f)\pi$ is given by 
$$\left(\widehat{f}(\pi)K_{\pi}^{\frac{1}{p^{\prime}}}\phi\right)(x)=\iint_{\R^2}(\F_1f)\left(x,\dfrac{\|y\|}{\|x\|},\cos^{-1}\left(\dfrac{x\cdot y}{\|x\|\|y\|}\right)\right)\|y\|^{\frac{2}{p^{\prime}}-1}\dfrac{\|x\|}{\|y\|^2}\phi(y)dy,$$
and the kernel 
\begin{equation} \label{kernel1}
		K^f(x,y)=
		\begin{cases}
			(\F_1f)\left(x,\dfrac{\|y\|}{\|x\|},\cos^{-1}\left(\dfrac{x\cdot y}{\|x\|\|y\|}\right)\right)\|y\|^{\frac{2}{p^{\prime}}-1}\dfrac{\|x\|}{\|y\|^2}, &~~ x\neq 0 , y\neq 0,\\
			0, &~~ otherwise.
		\end{cases}
	\end{equation}

Now \begin{align*}
\iint_{\R^4}|K^f(x,y)|^2dxdy &=\iint_{\R^4}\left|(\F_1f)\left(x,\dfrac{\|y\|}{\|x\|},\cos^{-1}\left(\dfrac{x\cdot y}{\|x\|\|y\|}\right)\right)\right|^2\|y\|^{\frac{4}{p^{\prime}}-2}\dfrac{\|x\|^2}{\|y\|^4}dxdy\\
& =\int_{\R^2}\int_0^{\infty}\int_0^{2\pi}|(\F_1f)(x,a,\theta)|^2(a\|x\|_2)^{\frac{4}{p^{\prime}}-2}\dfrac{dxdad\theta}{a^3}.
\end{align*}

For $0<\alpha<\frac{1}{2}$, let us choose $$f_{\alpha}(b,a,\theta)= 
\begin{cases}
|b_1|^{-\alpha}|b_2|^{-\alpha}a^3\theta,\quad (b,a,\theta)\in Q\\
0, \quad\quad \text{otherwise},
\end{cases}
$$
where $Q=\{(b,a,\theta):-p\leq b_1,b_2,a\leq p, \theta\in [0,2\pi)\}$.
Hence \begin{equation}\label{sim}
    (\F_1f_{\alpha})(\xi,a,\theta)=\left[ (2\pi)^{-1}\int_{-p}^p e^{-ib_1\xi_1}|b_1|^{-\alpha}db_1\int_{-p}^p e^{-ib_2\xi_2}|b_2|^{-\alpha}db_2\right]a^3\theta.
\end{equation}
Now for $\xi_1,\xi_2>0$,
\begin{equation}\label{sim2}
    \int_{-p}^p e^{-ib_j\xi_j}|b_j|^{-\alpha}db_j=2\left(\int_0^{\xi_j p}t^{-\alpha}\cos tdt\right)\xi_{j}^{-1+\alpha},\quad j=1,2.
\end{equation} When $\alpha\in (0,\frac{1}{2})$, then $|\int\limits_0^{\xi_j p}t^{-\alpha}\cos tdt|\geq A,$ for some $A>0$. Since $\dfrac{1}{\sqrt{2}}(|\xi_1|+|\xi_2|)\leq \|\xi\|_2\leq (|\xi_1|+|\xi_2|)$,  then $\|\xi\|_2\geq \dfrac{1}{\sqrt{2}}|\xi_1|$. Now using the equations \eqref{sim} and \eqref{sim}, we get
\begin{align*}
 \iint_{\R^4}|K^{f_{\alpha}}(x,y)|^2dxdy & =\int_{\R^2}\int_0^{\infty}\int_0^{2\pi}|(\F_1f_{\alpha})(\xi,a,\theta)|^2(a\|\xi\|_2)^{\frac{4}{p^{\prime}}-2}\dfrac{d\xi dad\theta}{a^3}\\
 &\geq (2\pi)^{-2}(2A)^4\int_0^{2\pi}\int_0^{\infty}\int_R^{\infty}\int_R^{\infty}\xi_1^{2(-1+\alpha)}\xi_2^{2(-1+\alpha)}a^3\theta \left(a|\xi_1|\right)^{\frac{4}{p^{\prime}}-2}\dfrac{d\xi dad\theta}{a^3}\\
 &= \dfrac{4A^4}{\pi^2}C_1^{\frac{4}{p^{\prime}}-2}\int_0^{2\pi}\theta d\theta\int_0^{\infty}a^{\frac{4}{p^{\prime}}-2}da\int_R^{\infty}\xi_1^{2(-1+\alpha)+{\frac{4}{p^{\prime}}-2}}d\xi_1\int_R^{\infty}\xi_2^{2(-1+\alpha)}d\xi_2.
\end{align*}
But $$\int_R^{\infty}\xi_1^{2(-1+\alpha)+{\frac{4}{p^{\prime}}-2}}d\xi_1=\infty,$$ when $\alpha>\frac{3}{2}-\frac{2}{p^{\prime}}$. Also $f_{\alpha}$ is square integrable if $\alpha<\frac{1}{2}$. Let us choose $\frac{1}{2}>\alpha>\frac{3}{2}-\frac{2}{p^{\prime}}, 2>p^{\prime}>\frac{4}{3}$, and  $\frac{1}{2}>\alpha>0>\frac{3}{2}-\frac{2}{p^{\prime}}, 1<p^{\prime}\leq\frac{4}{3}$, then $\widehat{f_{\alpha}}(\pi)K_{\pi}^{\frac{1}{p^{\prime}}}\notin S_2$. Thus for $2<p<\infty$,  $$\|\widehat{f_{\alpha}}(\pi)K_{\pi}^{\frac{1}{p^{\prime}}}\|_{S_{p^{\prime}}}=\infty, \frac{1}{p}+\frac{1}{p^{\prime}}=1,$$ where
$\frac{1}{2}>\alpha>\frac{3}{2}-\frac{2}{p^{\prime}}, 2>p^{\prime}>\frac{4}{3}$, and  $\frac{1}{2}>\alpha>0>\frac{3}{2}-\frac{2}{p^{\prime}}, 1<p^{\prime}\leq\frac{4}{3}$.

\section{Affine Poincar\'e Group}\label{s4}
This section covers the Fourier analysis of the affine Poincar\'e group, $\P$, in detail. In this setup, we will define the Weyl transform and show that it cannot be extended as a bounded operator for the symbol in the corresponding $L^p$ spaces with $2<p<\infty$.

The affine Poincar\'e group, $\P$, is a generalization of the affine group, or more precisely, a complexification of the affine group. This group emerges from the investigation of the $2$-dimensional wavelet transform. Similar to the four-parameter similitude group, $\P$ contains the translations $b$ in the image plane $\R^2$, global dilations (zooming in and out by $a>0$), but it has hyperbolic rotations around the origin ($\v\in\R$). The action on the plane is given by
$$x=(b,a,\v)y= a\Lambda_{\v}y+b,$$ where
	$b\in \mathbb{R}^2$, $a>0$, and $\Lambda_{\v}$ is the $2\times 2$ hyperbolic rotation matrix
	\begin{equation}
		\Lambda_{\v}=\left(\begin{matrix}
			\cosh\v &  \sinh\v\\
			\sinh\v   &  \cosh\v
		\end{matrix}\right).
	\end{equation}
	A convenient representation of the joint transformation $(b,a,\v)$ is in the form of $3\times 3$ matrices
	\begin{equation}
		(b,a,\v)=\left(\begin{matrix}
			a\Lambda_\v &  b\\
			0^{T}   &  1
		\end{matrix}\right),~~~ 0^{T}=(0,0).
	\end{equation}
Then the matrix multiplications gives the composition of successive transformations and thus the group law is derived as
	\begin{eqnarray}
		(b,a,\v)\ast (b^{\prime},a^{\prime},\v^{\prime})&=&(b+a\Lambda_{\v} b^{\prime},aa^{\prime},\v+\v^{\prime}),\nonumber
	\end{eqnarray}
	With respect to the operation $\ast$, $\P$ is a non-abelian group in which $({0},1,0)$ is the identity element and $(\frac{-1}{a}\Lambda_{-\v}{b},\frac{1}{a},-\v)$ is the inverse of $({b},a,\v)$ in $\P$. Also, it can be shown that $\P$ is a non-unimodular group as its left and right Haar measures   $$d\mu_{L}({b},a,\v)=\dfrac{d{b}dad\v}{a^3},~~ d\mu_{R}({b},a,\v)=\dfrac{d{b}dad\v}{a},$$ respectively are different, and hence the modular function is given by $\Delta(b,a,\v)=\dfrac{1}{a^2}$.
	Moreover, the affine Poincar\'e group $\P$ has the structure of a semi-direct product:
	$$\P=\mathbb{R}^{2}\rtimes(\mathbb{R}_{\ast}^{+}\times \mathrm{SO}(1,1)),$$ where $\mathbb{R}^2$ is the subgroup of the translations, $\mathbb{R}^{\ast}_{+}$ that of dilations, and $\mathrm{SO}(1,1)$ of hyperbolic rotations. Topologically, one can write
	$\P=\R^2\times\mathcal{C}$, where $\mathcal{C}$ is any one of the four cones:
	\begin{eqnarray}
		C_{1}^1=\{x\in\R^2:x_1^2>x_2^2, +x_1>0\} \\
		C_{2}^1=\{x\in\R^2:x_1^2>x_2^2,- x_1>0\}\\
		C_{1}^2=\{x\in\R^2:x_1^2<x_2^2,+ x_1>0\}\\
		C_{2}^2=\{x\in\R^2:x_1^2<x_2^2,- x_1>0\}.
	\end{eqnarray}
	 Let us define the Fourier transform $\F$ and inverse Fourier transform $\F^{-1}$,  by 
	 \begin{eqnarray}
	 	(\mathcal{F}\varphi)(\xi)=\frac{1}{2\pi}\int_{\R^2}e^{i\langle\xi;x\rangle}\varphi(x)dx \label{MinkowskiF}\\
	 	(\mathcal{F}^{-1}\varphi)(x)=\frac{1}{2\pi}\int_{\R^2}e^{-i\langle\xi;x\rangle}(\F{\varphi})(\xi)d\xi ,\label{inverseMinkowskiF} 
	 \end{eqnarray}   
for all $\varphi\in S(\R^n)$,	where $\langle x;y\rangle=x_1y_1-x_2y_2$ is the  Minkowski inner product. In this section, we denote $\F$ and $\F^{-1}$ are the Minkowski-Fourier transform and inverse Minkowski-Fourier transform on $\R^2$, respectively. 
\begin{rem}
	The relation between Euclidean Fourier transform and Minkowski-Fourier transform on $\R^2$ is $$(\mathcal{F}\varphi)(\xi_{1},\xi_2)=\widehat{\varphi}(-\xi_1,\xi_2), (\F^{-1}\varphi)(\xi_1,\xi_2)=\widecheck{\varphi}(-\xi_1,\xi_2),$$ where $\widehat{\varphi}$ and $\widecheck{\varphi}$ are the Euclidean Fourier transform  and inverse Euclidean Fourier transform $\varphi$ on $\R^2$, respectively. The reader can easily verify that the Minkowski-Fourier transform carries almost all properties similar to Euclidean Fourier transform on $\R^2$, like the Plancherel formula, and Parseval identity.  
\end{rem}
Define the mappings $\pi_i^j:\P\to\U(L^2(C_i^j)),$ by
$$\left( \pi_i^j(b,a,\v)\phi\right)(x)=ae^{\langle x;b\rangle}\phi(a\Lambda_{\v}x),$$ for all $(b,a,\v)\in\P$ and $\phi\in L^2(C_i^j), i,j=1,2.$ Let $f\in L^1(\P)\cap L^2(\P)$, the Fourier transform is defined by $$\left(\widehat{f}(\pi_i^j)\phi\right)(x)=\int_{\P}f(b,a,\v)\left(\pi_i^j(b,a,\v)\phi\right)(x)d\mu_L(b,a,\v),$$ for all $\phi\in L^2(C_i^j),i,j=1,2.$ Define the Duflo-Moore operators, \cite{San3}, $$(K_{i,j}\phi)(x)=\frac{1}{2\pi}|\langle x;x\rangle|\phi(x), i,j=1,2.$$ Define the $L^p$-Fourier transform $$\F_pf(\pi_i^j)=\widehat{f}(\pi_i^j)K_{i,j}^{\frac{1}{p^{\prime}}}, i,j=1,2,$$ where $\frac{1}{p}+\frac{1}{p^{\prime}}=1, 1\leq p\leq \infty$.

\begin{theorem}[Plancherel Theorem]\cite{San3}
 Let $f$ be in $L^1(\P)\cap L^2(\P)$. Then 
$$\int_{\P}|f(b,a,\v)|^2\dfrac{dbdad\theta}{a^3}=\sum_{i,j=1}^2\|\widehat{f}(\pi_i^j)K_{i,j}^{\frac{1}{2}}\|_{S_2}^2.$$   
\end{theorem}
\begin{theorem}[Inversion Theorem]\cite{San3} 
Let $f$ be in $L^1(\P)\cap L^2(\P)$. Then $$f(b,a,\v)=\sum_{i,j=1}^2\operatorname{Tr}\left(\pi_i^j(b,a,\v)^{\ast}\widehat{f}(\pi_i^j)K_{i,j}\right).$$   
\end{theorem}
For $G=\P$, recall the definition of Weyl transform defined in \eqref{Weyl1}, and the Theorem \ref{bounded} says that it is bounded when $1\leq p \leq 2$. Now we prove the unboundedness of Weyl transform when $p>2$.
\begin{theorem}\label{unbounded3}
For $2<p<\infty$, there exists an operator valued function $\sigma$ in $L^p(U\times \widehat{U}, S_p)$ such that $W_{\sigma}$ is not a bounded linear operator on $L^2(U)$. 
\end{theorem}
We prove the Theorem \ref{unbounded3} using the method of contrapositive in three propositions.
\begin{proposition}
Let $2<p<\infty$ and $\frac{1}{p}+\frac{1}{p^{\prime}}=1$. Then for all $\sigma\in L^p(\P\times\widehat{\P}, S_p)$, the Weyl transform $W_{\sigma}$ is bounded linear operator on $L^2(\P)$ iff there exists a constant $C$ such that $$\|W(f,g)\|_{p^{\prime},\mu}\leq C\|g\|_{L^2(\P)}\|f\|_{L^2(\P)},$$ for all $f,g$ in $L^2(\P)$.
\end{proposition}
\begin{proof}
The proof is followed from Theorem \ref{prop1}.
\end{proof}
\begin{proposition}
Let $2<p<\infty$ and $f$ be a square integrable, compactly supported function on $\P$ such that $\int\limits_{\P}f(b,a,\v)d\mu_L(b,a,\v)\neq 0$. If $W_{\sigma}$ is a bounded operator on $L^2(\P)$ for all $\sigma$ in $L^p(\P\times\widehat{\P},S_p)$, then 
$$\|\F_p(f)\pi_{1}^1\|_{S_{p^{\prime}}}^{p^{\prime}}+\|\F_p(f)\pi_{2}^2\|_{S_{p^{\prime}}}^{p^{\prime}}+\|\F_p(f)\pi_{2}^1\|_{S_{p^{\prime}}}^{p^{\prime}}+\|\F_p(f)\pi_{1}^2\|_{S_{p^{\prime}}}^{p^{\prime}}<\infty.$$
\end{proposition}
\begin{proof}
The proof is followed from Theorem \ref{prop2}.
\end{proof}
\begin{proposition}\label{example3}
For $p\in (2,\infty)$, does there exists a square-integrable, compactly supported function $f$ on $\P$ with $\int\limits_{\P} f(b,a,\vartheta)d\mu(b,a,\v)\neq 0$ such that $$\|\F_p(f)\pi_{1}^1\|_{S_{p^{\prime}}}^{p^{\prime}}+\|\F_p(f)\pi_{2}^2\|_{S_{p^{\prime}}}^{p^{\prime}}+\|\F_p(f)\pi_{2}^1\|_{S_{p^{\prime}}}^{p^{\prime}}+\|\F_p(f)\pi_{1}^2\|_{S_{p^{\prime}}}^{p^{\prime}}=\infty?$$
\end{proposition}
The answer to this question is yes. We will certainly find a function $f_{\alpha}, 0<\alpha<\frac{1}{2}$, at the end of the section, which is needed for Proposition \ref{example3}.
Let $f\in L^2(\P)$. For $i,j=1,2$,  consider the operators 
    \begin{align*}
        \left(\widehat{f}(\pi_i^j)K_{i,j}^{\frac{1}{p^{\prime}}}\phi\right)(x)&=\int_{\P}f(b,a,\v)\left(\pi_i^j(b,a,\v)K_{i,j}^{\frac{1}{p^{\prime}}}\phi\right)(x)\dfrac{dbdad\v}{a^3}\\
        &=\frac{1}{2\pi}\int_{\P}f(b,a,\v)ae^{i\langle x;b\rangle}|\langle a\Lambda_{-\v}x;a\Lambda_{-\v}x|^{\frac{1}{p^{\prime}}}\phi(a\Lambda_{-\v}x)\dfrac{dbdad\v}{a^3}\\
        &= \int_0^{\infty}\int_{\R}a^{1+\frac{2}{p^{\prime}}}(\F_1f)(x,a,\v)|\langle x;x|^{\frac{1}{p^{\prime}}}\phi(a\Lambda_{-\v}x)\dfrac{dad\v}{a^3},
    \end{align*}
    where $\phi\in L^2(C_i^j)$.
    Substitute $a\Lambda_{-\v}x=y$, then $a=\Bigg(\dfrac{\langle y;y\rangle}{\langle x;x\rangle}\Bigg)^{\frac{1}{2}}$,   $\v=\cosh^{-1}\Bigg(\dfrac{\langle x;y\rangle}{ \langle x;x\rangle}\times \Bigg(\dfrac{\langle x;x\rangle}{\langle y;y\rangle}\Bigg)^{\frac{1}{2}}\Bigg)$ and $dy=a|\langle x;x\rangle|dad\v.$

    Hence the integral representations become
 	$$ \left(\widehat{f}(\pi_i^j)K_{i,j}^{\frac{1}{p^{\prime}}}\phi\right)(x)=\int_{C_{i}^j} K_{i,j}^f(x,y)\phi(y)dy,$$
 	where
 	\begin{equation}
 	 K_{i,j}^f(x,y)=(\mathcal{F}_1f)\Bigg(x,\Bigg(\dfrac{\langle y;y\rangle}{\langle x;x\rangle}\Bigg)^{\frac{1}{2}}, \cosh^{-1}\Bigg(\dfrac{\langle x;y\rangle}{ \langle x;x\rangle}\times \Bigg(\dfrac{\langle x;x\rangle}{\langle y;y\rangle}\Bigg)^{\frac{1}{2}}\Bigg)\Bigg(\dfrac{\langle y;y\rangle}{\langle x;x\rangle}\Bigg)^{\frac{1}{2}+\frac{1}{p^{\prime}}}|\langle x;x\rangle|^{\frac{1}{p^{\prime}}}\dfrac{|\langle x;x\rangle|}{|\langle y,y\rangle|^2},
 		\end{equation}
 for all $(x,y)\in C_i^j\times C_i^j$, $i,j=1,2$. 
 	Now using the Plancherel formula for Minkowski-Fourier transform $\F$, we get
 	\begin{align}\label{int}
&\|\F_p(f)\pi_1^1\|_{S_2}^2+\|\F_p(f)\pi_1^2\|_{S_2}^2+\|\F_p(f)\pi_2^1\|_{S_2}^2+\|\F_p(f)\pi_2^2\|_{S_2}^2\nonumber\\	&=\sum_{i,j=1}^2\int_{C_{i}^j}\int_{C_{i}^j}|K_{i,j}^f(x,y)|^2dxdy\nonumber\\
 		&=\sum_{i,j=1}^2\int_{C_{i}^j}\int_{C_{i}^j}\left|{(\mathcal{F}_1f)\Bigg(x,\Bigg(\dfrac{\langle y;y\rangle}{\langle x;x\rangle}\Bigg)^{\frac{1}{2}}, \cosh^{-1}\Bigg(\dfrac{\langle x;y\rangle}{ \langle x;x\rangle}\times \Bigg(\dfrac{\langle x;x\rangle}{\langle y;y\rangle}\Bigg)^{\frac{1}{2}}\Bigg)}\right|^2\Bigg(\dfrac{\langle y;y\rangle}{\langle x;x\rangle}\Bigg)^{1+\frac{2}{p^{\prime}}}\nonumber\\
   &\times|\langle x;x\rangle|^{\frac{2}{p^{\prime}}}\dfrac{|\langle x;x\rangle|^2}{|\langle y,y\rangle|^4}dxdy\nonumber\\
 &=\sum_{i,j=1}^2\int_{C_{i}^j}\int_0^{\infty}\int_{\R}\left| (\mathcal{F}_1f)(x,a,\v)\right|^2 a^{1+\frac{2}{p^{\prime}}}|\langle x;x\rangle|^{\frac{2}{p^{\prime}}}\dfrac{|\langle x;x\rangle|^2}{|\langle y,y\rangle|^4}a|\langle x;x\rangle|dad\v dx\nonumber\\
 &=\sum_{i,j=1}^2\int_{C_{i}^j}\int_0^{\infty}\int_{\R}\left| (\mathcal{F}_1f)(x,a,\v)\right|^2 a^{\frac{2}{p^{\prime}}-6}|\langle x;x\rangle|^{\frac{2}{p^{\prime}}-1}dad\v dx\nonumber\\
 		&=\int_{\R^2}\int_0^{\infty}\int_{\R}\left| (\mathcal{F}_1f)(x,a,\v)\right|^2 a^{\frac{2}{p^{\prime}}-6}|\langle x;x\rangle|^{\frac{2}{p^{\prime}}-1}dxdad\v.
   \end{align}
Let us consider a function on $\R^2$, with $0<\alpha<\frac{1}{2}$,
$$
\phi_{\alpha}(b)=
\begin{cases}
  |b_1|^{-\alpha}|b_2|^{-\alpha}, \quad b\in Q,\\
  0, \quad\quad \text{otherwise},
\end{cases}
$$
 where $Q=\{b\in\R^2:-p\leq b_j\leq p, j=1,2\}$. 
 Then the Minkowski-Fourier transform of $\phi_{\alpha}$ becomes
 $$(\F\phi_{\alpha})(x)=\dfrac{1}{2\pi}\int_{-p}^{p}e^{b_1x_1}|b_1|^{-\alpha}db_1\int_{-p}^{p}e^{-b_2x_2}|b_2|^{-\alpha}db_2.$$ Now for $x_2>0,0<\alpha<\dfrac{1}{2}$, \begin{equation}\label{int1}
     \int_{-p}^{p}e^{-b_2x_2}|b_2|^{-\alpha}db_2=2\left(\int\limits_0^{px_2}t^{-\alpha}\cos tdt\right)x_2^{-1+\alpha}\geq 2Ax_2^{-1+\alpha},
 \end{equation} and for $x_1>0, 0<\alpha<\dfrac{1}{2}$, \begin{equation}\label{int2}
     \int_{-p}^{p}e^{-b_1x_1}|b_1|^{-\alpha}db_1=2\left(\int\limits_0^{px_1}t^{-\alpha}\cos tdt\right)x_1^{-1+\alpha}\geq 2Ax_1^{-1+\alpha}.
 \end{equation}
 Let us choose a function $f$ on $\P$, with $0<\alpha<\frac{1}{2}$, $$
 f_{\alpha}(b,a,\v)=
 \begin{cases}
     \phi_{\alpha}(b)a^{5}\v,\quad b\in Q, 0<a<p,-p\leq \v\leq p,\\
 0, \quad\quad \text{otherwise}.
 \end{cases}
 $$
 Now using equations \eqref{int1} and \eqref{int2}, the equation \eqref{int} becomes
 \begin{align}\label{iint}
&\|\F_p(f_{\alpha})\pi_1^1\|_{S_2}^2+\|\F_p(f_{\alpha})\pi_1^2\|_{S_2}^2+\|\F_p(f_{\alpha})\pi_2^1\|_{S_2}^2+\|\F_p(f_{\alpha})\pi_2^2\|_{S_2}^2\nonumber\\
&\geq 16A^4\int_R^{\infty}\int_R^{\infty}\int_0^p\int_{-p}^p x_1^{2(-1+\alpha)}x_2^{2(-1+\alpha)}a^{\frac{2}{p^{\prime}}-1}\v|\langle x;x\rangle|^{\frac{2}{p^{\prime}}-1}dxdad\v.
 \end{align}
 Let $A=\{(x_1,x_2): R<x_1,x_2<\infty\}$ and $B=\{(x_1,x_2)\in A:x_2<x_1-1\}$. It is easy to check that \begin{align*}x_1^2-x_2^2=x_1^2+x_2^2-2x_2^2
 \geq 2x_1x_2-2x_2^2=2x_2(x_1-x_2)\geq 2x_2, 
 \end{align*}
 on $B$.
 Consider the left side of equation \eqref{iint}, we have
 \begin{align*}
  \int_R^{\infty}\int_R^{\infty} x_1^{2(-1+\alpha)}x_2^{2(-1+\alpha)}|\langle x;x\rangle|^{\frac{2}{p^{\prime}}-1}dx & \geq  \iint_B x_1^{2(-1+\alpha)}x_2^{2(-1+\alpha)}(2x_2)^{\frac{2}{p^{\prime}}-1}dx \\
  &=\int_{x_1=R}^{\infty}\int_{x_2=R}^{x_2=x_1-1}x_1^{2(-1+\alpha)}x_2^{2(-1+\alpha)}(2x_2)^{\frac{2}{p^{\prime}}-1}dx_1dx_2\\
   &= 2^{\frac{2}{p^{\prime}}-1}\int_{x_1=R}^{\infty}\int_{x_2=R}^{x_2=x_1-1}x_1^{2(-1+\alpha)}x_2^{2(-1+\alpha)+\frac{2}{p^{\prime}}-1}dx_1dx_2\\
   &= K\int_{x_1=R}^{\infty}x_1^{2(-1+\alpha)}(x_1-1)^{2(-1+\alpha)+\frac{2}{p^{\prime}}}dx_1\\
   &-KR^{2(-1+\alpha)+\frac{2}{p^{\prime}}}\int_{R}^{\infty}x_1^{2(-1+\alpha)}dx_1,
 \end{align*}
 where $K=\dfrac{1}{2(-1+\alpha)+\frac{2}{p^{\prime}}}$. The integral $$\int_{x_1=R}^{\infty}x_1^{2(-1+\alpha)}(x_1-1)^{2(-1+\alpha)+\frac{2}{p^{\prime}}}dx_1=\infty,$$ when $\frac{1}{2}>\alpha\geq\frac{3}{4}-\frac{1}{2p^{\prime}}>0, 1<p^{\prime}<2$, and the integral  $\int\limits_{R}^{\infty}x_1^{2(-1+\alpha)}dx_1<\infty, $ when $\alpha<\frac{1}{2}$. Then the left-hand side of equation \eqref{iint} becomes infinity, and hence  $$\|\F_p(f_{\alpha})\pi_1^1\|_{S_2}^2+\|\F_p(f_{\alpha})\pi_1^2\|_{S_2}^2+\|\F_p(f_{\alpha})\pi_2^1\|_{S_2}^2+\|\F_p(f_{\alpha})\pi_2^2\|_{S_2}^2=\infty,$$ when $\frac{1}{2}>\alpha\geq\frac{3}{4}-\frac{1}{2p^{\prime}}>0, 1<p^{\prime}<2$. Thus for these values of $\alpha$,
 $$\|\F_p(f_{\alpha})\pi_{1}^1\|_{S_{p^{\prime}}}^{p^{\prime}}+\|\F_p(f_{\alpha})\pi_{2}^2\|_{S_{p^{\prime}}}^{p^{\prime}}+\|\F_p(f_{\alpha})\pi_{2}^1\|_{S_{p^{\prime}}}^{p^{\prime}}+\|\F_p(f_{\alpha})\pi_{1}^2\|_{S_{p^{\prime}}}^{p^{\prime}}=\infty.$$

\section*{acknowledgment} The author is thankful to Prof. Aparajita Dasgupta for many fruitful discussions on the problem.


\begin{thebibliography}{amsplain}
		
		
		\normalsize
		\baselineskip=17pt

       \bibitem{Boggiatto1} Boggiatto, B. Buzano, E. and Rodino, L. \emph{Global ellipticity and spectral theory}. Mathematical Research, 92(1996).
       
     \bibitem{Boggiatto2} Boggiatto, P. and Rodino, L. ``Quantization and pseudo-differential operators." \emph{Cubo Mathematical Eductional}, \emph{1}(5), (2003).
       
		
	
		\bibitem{Chen} Chen, L. and Zhao, J. ``Weyl transform and generalized spectrogram associated with quaternion Heisenberg group." \emph{Bulletin des Sciences Mathématiques}, \emph{136}(2), 127--143(2012).
		
		


\bibitem{San1}  Dasgupta, A. Nayak, S. K., and Wong M. W. ``Hilbert-Schmidt and Trace Class Pseudo-Differential Operators and Weyl Transforms on the Affine Group." \emph{Journal of Pseudo-differential Operators}, \emph{12}(11), (2021).

\bibitem{San2} Dasgupta, A. and  Nayak, S. K. ``Pseudo-Differential Operators, Wigner Transform and Weyl Transforms on Similitude Group, $\mathbb{SIM}(2)$",  \emph{Bulletin des sciences mathématiques}, Volume 174, 103087, (2022).

\bibitem{San3}Dasgupta, A. and  Nayak, S. K.  ``Pseudo-Differential Operators, Wigner Transform and Weyl Transform on the Affine Poincar\'e Group." \emph{Bulletin des sciences mathématiques}, Volume 184, 103255, (2023).

\bibitem{San4} Dasgupta, A. and Nayak, S. K. ``Localization Operator and Weyl Transform on Reduced Heisenberg Group with Multi-dimensional Center." 2022. http://arxiv.org/abs/2211.06613 

  
		
		
		
		

\bibitem{Fuhr} F\"uhr, H. \emph{Abstract harmonic analysis of continuous wavelet transforms}. vol. 1963. Springer
Science \& Business Media, (2005).
  

\bibitem{Ghosh} Ghosh, S. and Srivastav, R. K. ``{Unbounded Weyl transforms on the Euclidean motion group and Heisenberg motion group}." DOI: arXiv:2106.15704.
  

  
		\bibitem{Kohn} Kohn, J. J. and Nirenberg, L. ``An algebra of pseudo-differential operators." \emph{Communications on Pure and Applied mathematics}, \emph{18}(1), 269--305(1965).
		
		
		
		\bibitem{Peng} Peng, L. and Zhao, J. `` Weyl transforms associated with the Heisenberg group." \emph{Bulletin des Sciences Mathématiques}, \emph{132}, 78--86(2008).
		
		\bibitem{Peng1} Peng, L. and Zhao, J. ``Weyl transforms on the upper half plane." \emph{Revista Matemática Complutense}, \emph{23}, 77--95(2010).
		
		
		\bibitem{Rachdi} Rachdi, L. T. and Trim\'eche, K. ``Weyl transforms associated with the spherical mean operator." \emph{Analysis and Applications}, \emph{1}(2), 141--164(2003).
		
		\bibitem{Simon} Simon, B. `` The Weyl transforms and $L^p$ functions on phase space." \emph{Proceedings of the American Mathematical Society}, \emph{116}(4), 1045--1047(1992).
		
		\bibitem{stien} Stien, E. M. and Murphy, T. S. \emph{Harmonic Analysis: Real-Variable Methods, Orthogonality, and Oscillatory Integrals}. Princeton University Press, New Jersey(1993).
		
      \bibitem{Toft1} Toft, J. ``Continuity properties for modulation spaces with applications to pseudo-differential calculus I." \emph{J. Funct. Anal.},\emph{207}(2), 399--429(2004).
      

      \bibitem{Toft2} Toft, J. ``Continuity properties for modulation spaces with applications to pseudo-differential calculus II." \emph{Ann. Global Anal. Geom.},\emph{26}, 73--106(2004).
      
  
		\bibitem{Weyl} Weyl, H. \emph{The theory of groups and quantum mechanics}. Dover(1950).
		
		\bibitem{wong1} Wong, M. W. \emph{The Weyl Transform}. Springer(1998).
		
		
		\bibitem{Zhao} Zhao, J. M. and Peng, L. Z. ``Wavelet and Weyl transform associated with the spherical mean operator." \emph{Integral Equation and Operator Theory}, \emph{50}, 279--290(2004).		
		
	\end{thebibliography}
\end{document}